\documentclass{amsart}

\RequirePackage[abspath,realmainfile]{currfile}

\usepackage{amsmath,amssymb,amsthm,amsfonts,amsbsy}
\usepackage{graphicx}
 
\usepackage{enumerate}
\usepackage[final]
{showkeys}

\usepackage{datetime} 
\usepackage{xcolor}

\usepackage{hyperref}

\providecommand{\url}[1]{#1}

\renewcommand{\le}{\leqslant}
\renewcommand{\ge}{\geqslant}

\newcommand{\intr}[2]{\overline{#1,#2}}

\renewcommand{\P}{\operatorname{\mathsf{P}}} 
\newcommand{\E}{\operatorname{\mathsf{E}}}

\newcommand{\ii}[1]{\operatorname{\mathsf{I}}\{#1\}}
\newcommand{\iii}{\operatorname{\mathsf{I}}}

\DeclareMathOperator{\vol}{\operatorname{vol}}

\newcommand{\bi}{\binom}

\newcommand{\R}{\mathbb{R}}                  
                 
\newcommand{\Z}{\mathbb{Z}}

\newcommand{\al}{\alpha}
\newcommand{\be}{\beta}
\newcommand{\ga}{\gamma}

\newcommand{\vp}{\varepsilon}

\newcommand{\si}{\sigma}

\newtheorem*{theorem*}{Theorem~\ref{th:}'}
\newtheorem*{thm*}{Theorem~\ref{th:der}'}

\newtheorem{theorem}{Theorem}
\newtheorem{corollary}[theorem]{Corollary}
\newtheorem*{corollary*}{Corollary~\ref{cor:A}'}
\newtheorem*{cor*}{Corollary~\ref{cor:Ader}'}

\theoremstyle{remark}
\newtheorem{remark}[theorem]{Remark}

\numberwithin{equation}{section}
\numberwithin{theorem}{section}

\begin{document}

\title[Order statistics on the spacings between order statistics]{Order statistics on the spacings between order statistics for the uniform distribution}

\author{Iosif Pinelis}
\address{Department of Mathematical Sciences\\Michigan Technological University\\Hough\-ton, Michigan 49931}
\email{ipinelis@mtu.edu}




\date{\today}

\begin{abstract}
Closed-form expressions for the distributions of the order statistics on the spacings between order statistics for the uniform distribution are obtained. This generalizes a result by Fisher concerning tests of significance in the harmonic analysis of a series. 
\end{abstract}

\maketitle


\tableofcontents

\section{
Summary 
and discussion}
\label{intro}

For any natural $n
$, let $U_1,\dots,U_n$ be independent 
random variables (r.v.'s) each uniformly distributed on the interval $[0,1]$.  As usual, let $U_{n:1}\le\cdots\le U_{n:n}$ denote the corresponding order statistics. Consider  
\begin{equation}\label{eq:G_i}
	G_i:=U_{n:i}-U_{n:i-1} \quad \text{for}\quad i\in\intr1{n+1}, 
\end{equation}
where $U_{n:0}:=0$ and $U_{n+1:n+1}:=1$. 
Here and in what follows, $\intr\al\be:=\{k\in\Z\colon\al\le k\le\be\}$. 

One may refer to the $G_i$'s as the gaps or, as it is usually done in the literature, the spacings between the consecutive order statistics. See e.g.\ the paper by Pyke \cite{pyke65}, containing a review of known results for the spacings for the underlying uniform distribution and other distributions as well; see also \cite{david-nagaraja} for later updates. 

Let now 
$
	G_{n+1:1}\le\cdots\le G_{n+1:n+1}
$ 
denote the ordered gaps $G_1,\dots,G_{n+1}$, so that the ransom sets $\{G_{n+1:1},\dots, G_{n+1:n+1}\}$ and $\{G_1,\dots,G_{n+1}\}$ are the same. Let also $G_{n+1:0}:=0$ and $G_{n+1:n+2}:=1$, so that 
\begin{equation}\label{eq:G_n+1:i}
	0=G_{n+1:0}\le G_{n+1:1}\le\cdots\le G_{n+1:n+1}\le G_{n+1:n+2}=1. 
\end{equation}

The main result of this note describes the cumulative distribution function (cdf) of each of the ordered gaps $G_{n+1:1},\dots, G_{n+1:n+1}$: 


\begin{theorem}\label{th:}
For all $k\in\intr1{n+1}$ and $x\in[0,1]$ 
\begin{equation}\label{eq:}
	\P(G_{n+1:k}>x)=(-1)^{k+1}(n+1)\bi n{k-1}
	\sum_{r=0}^{k-1}\frac{(-1)^r}{n-r+1}\bi{k-1}r\big(1-(n-r+1)x\big)_+^n.  
\end{equation}
\end{theorem}

Everywhere here, $u_+:=0\vee u=\max(0,u)$. 

The proof of Theorem~\ref{th:} is based on 

\begin{theorem}\label{th:2}
Take any $k\in\intr0{n+1}$ and $x\in(0,1)$. Then  
\begin{equation}\label{eq:2}
	\P(G_{n+1:k}\le x<G_{n+1:k+1})=(-1)^k\bi{n+1}k
	\sum_{r=0}^k(-1)^r\bi kr\big(1-(n-r+1)x\big)_+^n.  
\end{equation}
\end{theorem}

In particular, choosing $k=n+1$ in \eqref{eq:} or \eqref{eq:2}, we immediately get 
\begin{equation}\label{eq:k=n+1}
	\P(G_{n+1:n+1}>x)=
	\sum_{s=1}^{n+1}(-1)^{s-1}\bi{n+1}s\big(1-sx\big)_+^n.  
\end{equation}

\begin{remark}\label{rem:exp}
It 
has now been long  
a textbook fact (see e.g.\ \cite[Exercise~20, page 103]{arnold-etal92}) that the joint distribution of the gaps $G_1,\dots,G_{n+1}$ is the same as that of $R_1,\dots,R_{n+1}$, where 
\begin{equation}\label{eq:R_i}
R_i:=\frac{X_i}{X_1+\dots+X_{n+1}}	
\end{equation}
and the $X_i$'s are independent (say standard) exponential random variables. 
Moran \cite[page~93]{moran47} ascribes mentioning of this fact to Fisher \cite{fisher29}, and a proof of it -- without a specific reference -- to Clifford. 
\end{remark}

In fact, Fisher \cite{fisher29} used geometric arguments to obtain the following formula for the distribution of $R_{n+1:n+1}=\max_{1\le i\le n+1}R_i$: 
\begin{equation}\label{eq:fisher}
	\P(R_{n+1:n+1}>x)=\sum_{j=1}^{n+1}(-1)^{j-1}\binom{n+1}j(1-jx)_+^n. 
\end{equation}
In view of Remark~\ref{rem:exp}, \eqref{eq:k=n+1} is equivalent to \eqref{eq:fisher}. 

Accordingly, one may replace $G_{n+1:k}$ and $G_{n+1:k+1}$ in \eqref{eq:} and \eqref{eq:2} by $R_{n+1:k}$ and $R_{n+1:k+1}$, and thus obtain generalizations of \eqref{eq:fisher}. 

In Fisher's setting, r.v.'s of the form $cX_i$ with an unknown real parameter $c>0$ were certain test statistics in independent tests of significance in the harmonic analysis of a series; these tests are labeled by $1,\dots,n+1$ in this note and by $1,\dots,n$ in \cite{fisher29}. 
To remove the unknown parameter $c$, the test statistics $cX_i$ were then normalized in \cite{fisher29} by the observable quantity $cX_1+\dots+cX_{n+1}$, yielding the ratios $R_i$, as in \eqref{eq:R_i}, with a known joint distribution. 
Thus, $\P(R_{n+1:n+1}>x)$ is the probability that at least one of the $n+1$ normalized test statistics $R_i$ will exceed the critical value $x$. 

Accordingly, the probability $\P(G_{n+1:k}>x)$ in \eqref{eq:}, which equals $\P(R_{n+1:k}>x)$, is the probability that at least $\ell:=(n+1)-(k-1)=n-k+2$ of the $n+1$ normalized test statistics $R_i$ will exceed the critical value $x$; note that here $\ell$ can take any value in $\intr1{n+1}$. 
Similarly, the probability $\P(G_{n+1:k}\le x<G_{n+1:k+1})$ in \eqref{eq:2} is the probability that exactly $m:=n+1-k$ of the $n+1$ normalized test statistics $R_i$ will exceed the critical value $x$; note that here $m$ can take any value in $\intr0{n+1}$.  

Thus, Theorems~\ref{th:} and \ref{th:2} above provide useful additional information concerning the independent tests considered by Fisher. 

The proofs of Theorems~\ref{th:} and \ref{th:2} are based on certain geometric, combinatorial, and 
analytic considerations. 
The method of proof of \eqref{eq:fisher} in \cite{fisher29} does not seem to work for the more general results presented in the present note. 

The following result is also based on Theorem~\ref{th:2}. 
\begin{corollary}\label{cor:}
Take any $k\in\intr0{n+1}$. Then 
\begin{equation}\label{eq:EG}
	\E G_{n+1:k}=\frac{H_{n+1}-H_{n+1-k}}{n+1}, 
\end{equation}
where 
\begin{equation*}
	H_j:=1+\frac12+\dots+\frac1j 
\end{equation*}
is the $j$th harmonic number, with $H_0:=0$. 
\end{corollary}

In particular, since $\frac1r\sim\int_r^{r+1}\frac{dx}{x}$ as $r\to\infty$, it follows from \eqref{eq:EG} that 
\begin{equation*}
	\E G_{n+1:k}\sim\frac1n\,\ln\frac n{n-k}\quad\text{if}\quad n-k\to\infty; 
\end{equation*}
as usual, we write $a\sim b$ for $a/b\to1$. 
Further, if $k=o(n)$, then  
\begin{equation*}
	\E G_{n+1:k}\sim\frac k{n^2},  
\end{equation*}
so that $\E G_{n+1:k}$ is asymptotically linear in $k$. 
Further, taking $k=1$, we also see that 
the expectation of the smallest among the gaps $G_1,\dots,G_{n+1}$ is $\E G_{n+1:1}=\frac1{(n+1)^2}$, which is $n+1$ times as small as the average $\frac1{n+1}$ of these $n+1$ gaps. 

On the other hand, the expectation of the largest among the gaps $G_1,\dots,G_{n+1}$ is 
\begin{equation*}
	\E G_{n+1:n+1}=\frac{H_{n+1}}{n+1}\sim\frac{\ln n}n
\end{equation*}
as $n\to\infty$, so that the largest gap is about $\ln n$ times as large on the average as the average of the gaps. 

\section{Proofs}\label{proofs}

\begin{proof}[Proof of Theorem~\ref{th:2}]
We begin with the following simple observation. Let $\mu$ be a measure defined on the Borel $\si$-algebra over $\R^n$ with a finite joint tail function $T_\mu$ defined by the formula 
\begin{equation*}
	T_\mu(x):=\mu\big(Q(x)\big), 
\end{equation*}
where 
\begin{equation*}
	Q(x):=\prod_{i=1}^n(x_i,\infty)
\end{equation*}
is the ``tail'' orthant with the vertex $x=(x_1,\dots,x_n)\in\R^n$. 

For any $a=(a_1,\dots,a_n)$ and $b=(b_1,\dots,b_n)$ in $\R^n$ such that $a_i\le b_i$ for all $i\in\intr1n$, consider the parallelepiped  
\begin{equation*}
	\Pi_{a,b}:=\prod_{i=1}^n(a_i,b_i]. 
\end{equation*}
Also, let 
$$h=(h_1,\dots,h_n):=b-a$$ 
and, for each $\vp=(\vp_1,\dots,\vp_n)\in\{0,1\}^n$, let $a+\vp h:=(a_1+\vp_1 h_1,\dots,a_n+\vp_n h_n)$ and $|\vp|:=\vp_1+\dots+\vp_n$; note that the $i$th coordinate $a_i+\vp_i h_i$ of the vector $a+\vp h$ equals $a_i$ or $b_i$ depending on whether $\vp_i$ equals $0$ or $1$. As usual, let $\iii_A$ denote the indicator function of a set $A$. 
Then 
\begin{align*}
	\mu(\Pi_{a,b})
	&=\int_{\R^n}d\mu\,\prod_{i=1}^n\big(\iii_{(a_i,\infty)}-\iii_{(b_i,\infty)}\big) \\ 	&=\int_{\R^n}d\mu\,	
\sum_{\vp\in\{0,1\}^n}(-1)^{|\vp|}\,\iii_{Q(a+\vp h)},
\end{align*}
whence
\begin{equation}\label{eq:mu(Pi)}
	\mu(\Pi_{a,b})=
\sum_{\vp\in\{0,1\}^n}(-1)^{|\vp|}\,T_\mu(a+\vp h). 
\end{equation}

In particular, for $y\in\R$, $\al\in[0,\infty)$, and $\ga=(\ga_1,\dots,\ga_n)\in(0,\infty)^n$, consider now the measure $\mu_{y,\al,\ga}$ that has the density with respect to the Lebesgue measure on $\R^n$ given by the formula 
\begin{equation*}
	\frac{d\mu_{y,\al,\ga}}{dx}=(y-\ga\cdot x)_+^\al
\end{equation*}
for $x=(x_1,\dots,x_n)\in\R^n$, where $\ga\cdot x:=\sum_1^n\ga_i x_i$ and (concerning the case $\al=0$) $0^0:=0$. Then, using induction on $n$ or, more specifically, iterated integration, it is easy to see that 
\begin{equation*}
	T_{\mu_{y,\al,\ga}}(x)= \frac{(y-\ga\cdot x)_+^{\al+n}}{\prod_1^n\big((\al+i)\ga_i\big)}
\end{equation*}
for $x=(x_1,\dots,x_n)\in\R^n$. 

Choosing now $\al=0$, we get 
\begin{equation}\label{eq:vol}
	\vol_n(\Pi_{a,b}\cap H_{\ga,y})
	=\frac1{n!\prod_1^n\ga_i}\,\sum_{\vp\in\{0,1\}^n}(-1)^{|\vp|}\,
	\big(y-\ga\cdot(a+\vp h)\big)_+^n,  
\end{equation}
where $\vol_n$ denotes the volume in $\R^n$ (that is, the Lebesgue measure on $\R^n$) and 
$
H_{\ga,y}:=\big\{x\in\R^n\colon \ga\cdot x\le y\big\}.  	
$ 
For $a=(0,\dots,0)$ and $b=(1,\dots,1)$, formula \eqref{eq:vol} was given in \cite{barrow-smith} and \cite{lawrence}. (The condition that $\ga_i>0$ for all $i$ was missing in \cite{barrow-smith}.) 

Next, note that 
the joint probability density function (pdf), say $f$, of the order statistics $U_{n:1},\dots,U_{n:n}$ is given by the formula 
\begin{equation}\label{eq:f}
	f(y_1,\dots,y_n)=n!\,\ii{0<y_1<\dots<y_n<1}
\end{equation}
for $(y_1,\dots,y_n)\in\R^n$; see e.g.\ \cite[page~12]{david-nagaraja}. 
Since, in view of \eqref{eq:G_i}, the r.v.'s $G_1,\dots,G_n$ are obtained from $U_{n:1},\dots,U_{n:n}$ by a linear transformation with determinant $1$, we see that 
the joint pdf, say $g$, of the gaps $G_1,\dots,G_n$ is given by the formula 
\begin{equation}\label{eq:g}
	g(z_1,\dots,z_n)=n!\,\ii{z_1>0,\dots,z_n>0,z_1+\dots+z_n<1}
\end{equation}
for $(z_1,\dots,z_n)\in\R^n$. 

Take now any $j\in\intr0{n
}$, $x\in(0,1)$, and $y\in(0,1]$, and let
\begin{equation}\label{eq:p_n}
	p_{n,j}(x,y):=\P\Big(G_i\le x\ \forall i\in\intr1j,\ 
	G_i>x\ \forall i\in\intr{j+1}n,\ \sum_1^n G_i<y\Big).
\end{equation}
Then, by \eqref{eq:g}, 
\begin{equation*}
	p_{n,j}(x,y)=n!\vol_n(
	\Pi_{a^{j,x},b^{j,x}}\cap H_{\ga_1,y}),
\end{equation*}
where 
$a^{j,x}_i:=0$ and $b^{j,x}_i:=x$ for $i\in\intr1j$, $a^{j,x}_i:=x$ and $b^{j,x}_i:=1$ for $i\in\intr{j+1}n$, and $\ga_1:=(1,\dots,1)\in\R^n$. 

So, letting $|\vp|_*:=\sum_1^j\vp_i$ and $|\vp|_{**}:=\sum_{j+1}^n\vp_i$ for $\vp\in\{0,1\}^n$ and using \eqref{eq:vol}, we have 
\begin{align*}
	p_{n,j}(x,y)&=\sum_{\vp\in\{0,1\}^n}(-1)^{|\vp|}\,
	\big(y-(n-j)x-|\vp|_* x-|\vp|_{**}(1-x)\big)_+^n \\ 
	&=\sum_{\al=0}^j\sum_{\be=0}^{n-j}(-1)^{\al+\be}\,\bi j\al \bi{n-j}\be 
	\big(y-(n-j)x-\al x-\be(1-x)\big)_+^n. 
\end{align*}
Note also that $(n-j-\be+\al)x\ge0$ for $\al\in\intr0j$, $\be\in\intr0{n-j}$, and $x\in(0,1)$, so that  
$\big(y-(n-j)x-\al x-\be(1-x)\big)_+=\big(y-\be-(n-j-\be+\al)x\big)_+=0$ for $y\in(0,1]$ and $\be\in\intr1{n-j}$. 
Therefore, the latter displayed expression for $p_{n,j}(x,y)$ greatly simplifies: 
\begin{align}
	p_{n,j}(x,y)&=\sum_{\al=0}^j(-1)^\al\,\bi j\al \big(y-(n-j+\al)x\big)_+^n \notag \\ 
	&=\sum_{r=0}^j (-1)^{j-r} \bi jr \big(y-(n-r)x\big)_+^n \notag \\ 
	&=\sum_{r=-\infty}^\infty (-1)^{j-r} \bi jr \big(y-(n-r)x\big)_+^n. \label{eq:p_n,k=} 
\end{align}
The latter equality holds because 
\begin{equation}\label{eq:bin}
\bi jr=0\quad\text{for}\quad j\in\intr0\infty\quad\text{and}\quad r\in\intr{-\infty}{-1}\,\cup\,\intr{j+1}\infty,
\end{equation}
with $\bi jr$ understood in the combinatorial sense, as the cardinality of the set $\bi{[j]}r$ of all subsets of cardinality $r$ of the set 
$$[j]:=\intr1j.$$  
for instance, for all $j\in\intr0\infty$ we have $\bi j{-1}=0$ because 
$\bi{[j]}{-1}=\emptyset$. For another, analytic approach to generalized binomial coefficients $\bi jr$, which leads to the same results for $j\in\intr0\infty$ and $r\in\intr{-\infty}{-1}\,\cup\,\intr{j+1}\infty$, see e.g.\ \cite{fowler}.

After these preliminary observations, we are ready to consider the probability in \eqref{eq:2}: 
\begin{equation}\label{eq:P_n,k}
	P_{n,k}:=\P(G_{n+1:k}\le x<G_{n+1:k+1})=Q_{n,k}+R_{n,k},  
\end{equation}
where 
\begin{align*}
	Q_{n,k}&:=\sum_{J\in\bi{[n]}k}\P\Big(G_i\le x\ \forall i\in J,\ 
	G_i>x\ \forall i\in[n]\setminus J,\ \sum_1^n G_i<1-x\Big) \\ 
	&=\bi nk	p_{n,k}(x,1-x), \\ 
	R_{n,k}&:=\sum_{J\in\bi{[n]}{k-1}}\P\Big(G_i\le x\ \forall i\in J,\ 
	G_i>x\ \forall i\in[n]\setminus J,\ 1>\sum_1^n G_i\ge1-x\Big) \\ 
	&=\bi n{k-1}	\big(p_{n,k-1}(x,1)-p_{n,k-1}(x,1-x)\big);   
\end{align*}
here we used the definition of $p_{n,k}(x,y)$ in \eqref{eq:p_n} and the fact that, in view of \eqref{eq:g}, the r.v.'s $G_1,\dots,G_n$ are exchangeable. 

In Theorem~\ref{th:2}, $k$ may take any value in the set $\intr0{n+1}$, whereas the expression in \eqref{eq:p_n,k=} for $p_{n,j}(x,y)$ was established only for $j\in\intr0n$. However, $Q_{n,n+1}=0$ because $\bi n{n+1}=0$, and $R_{n,0}=0$ because $\bi n{-1}=0$, whereas $k-1\in\intr0n$ for $k\in\intr1{n+1}$. It follows that, for all $k\in\intr0{n+1}$, we can replace all entries of $p_{n,\cdot}(x,\cdot)$ in the above expressions for $Q_{n,k}$ and $R_{n,k}$ by the corresponding expressions according to \eqref{eq:p_n,k=}.  
Thus, letting now 
\begin{equation}\label{eq:a_r}
	a_r:=a_{n,r}(x):=(-1)^r\big(1-(n-r)x\big)_+^n,  
\end{equation}
we have 
\begin{align}
	(-1)^{k+1}P_{n,k}&=\bi nk\sum_{r=-\infty}^\infty \bi kr a_{r-1} \notag \\ 
	&+\bi n{k-1}\sum_{r=-\infty}^\infty \bi {k-1}r a_r \notag \\ 
	&+\bi n{k-1}\sum_{r=-\infty}^\infty \bi {k-1}r a_{r-1} \notag \\ 
	&=\sum_{r=-\infty}^\infty
	\left(\bi nk\bi kr+\bi n{k-1}\bi{k-1}{r-1}+\bi n{k-1}\bi{k-1}r\right)a_{r-1} \notag \\ 
	&=\sum_{r=-\infty}^\infty
	\left(\bi nk\bi kr+\bi n{k-1}\bi kr\right)a_{r-1} \notag \\ 
	&=\sum_{r=-\infty}^\infty \bi{n+1}k \bi kr a_{r-1}
	=\bi{n+1}k\sum_{r=-\infty}^\infty \bi kr a_{r-1}. \label{eq:P_n,k=}
\end{align}
Now \eqref{eq:2} immediately follows, in view of \eqref{eq:P_n,k} and \eqref{eq:a_r}. 
\end{proof}

\begin{proof}[Proof of Theorem~\ref{th:}]
For all $j\in\intr0{n+1}$ 
\begin{align}
	\P(G_{n+1:j+1}>x)-\P(G_{n+1:j}>x)
	&=\P(G_{n+1:j}\le x<G_{n+1:j+1}) \notag \\ 
	&=P_{n,j}=(-1)^{j+1}\bi{n+1}j\sum_{r=0}^j\bi jr a_{r-1}, \label{eq:P_n,j=}
\end{align}
the latter two equalities holding by virtue of \eqref{eq:P_n,k} and \eqref{eq:P_n,k=}. 
Also, in view of \eqref{eq:G_n+1:i} and because $x\in(0,1)$, we have $\P(G_{n+1:0}>x)=0$. So, one can find $\P(G_{n+1:k}>x)$ by summation: 
\begin{equation*}
	\P(G_{n+1:k}>x)=\sum_{j=0}^{k-1}P_{n,j}, 
\end{equation*}
and this is how the expression of $\P(G_{n+1:k}>x)$ in \eqref{eq:} was actually found. 

However, once that expression has been obtained, it is sufficient -- and much easier -- to verify \eqref{eq:} by checking the identity 
\begin{equation}\label{eq:V-V}
	V_{n,j+1}-V_{n,j}\overset{\text{(?)}}=P_{n,j}
\end{equation}
for all $j\in\intr0n$, where 
\begin{equation}\label{eq:V_n,j}
	V_{n,j}:=(-1)^j(n+1)\bi n{j-1}
	\sum_{r=0}^{j-1}\frac{a_{r-1}}{n-r+1}\bi{j-1}r, 
\end{equation}
the right-hand side of \eqref{eq:} with $j$ in place of $k$, taking also \eqref{eq:a_r} into account;  
hence, 
\begin{equation}\label{eq:V_n,j+1}
	V_{n,j+1}:=(-1)^{j+1}(n+1)\bi nj
	\sum_{r=0}^j\frac{a_{r-1}}{n-r+1}\bi jr.  
\end{equation}

Indeed, it will immediately follow from \eqref{eq:P_n,j=} and \eqref{eq:V-V} that 
\begin{equation*}
	\P(G_{n+1:j+1}>x)-\P(G_{n+1:j}>x)=V_{n,j+1}-V_{n,j}
\end{equation*}
for $j\in\intr0n$. Since $\P(G_{n+1:0}>x)=0=V_{n,0}$, it will then follow by induction on $k$ or, equivalently, by telescoping summation, that 
$\P(G_{n+1:k}>x)=V_{n,k}$ for all $k\in\intr0{n+1}$, which will complete the proof of Theorem~\ref{th:}. 

Turning now back to identity \eqref{eq:V-V}, we see 
that each side of it equals $-a_{-1}$ when $j=0$. 

Next, it is convenient to replace $\sum_{r=0}^j$ in \eqref{eq:P_n,j=} and \eqref{eq:V_n,j+1}, as well as $\sum_{r=0}^{j-1}$ in \eqref{eq:V_n,j}, by $\sum_{r=0}^n$; in view of \eqref{eq:bin} and the condition $j\in\intr0n$, these replacements will not affect the values of the corresponding expressions for $P_{n,j}$, $V_{n,j+1}$, and $V_{n,j}$. 

Thus, it suffices to check that the coefficients of the $a_{r-1}$'s on both sides of \eqref{eq:V-V} are the same for all $j\in\intr1n$ and $r\in\intr0n$, which amounts to checking the identity 
\begin{equation*}
	\frac{n+1}{n-r+1}\left(\bi nj \bi jr+\bi n{j-1} \bi{j-1}r\right)
	\overset{\text{(?)}}=\bi{n+1}j \bi jr
\end{equation*}
for such $j$ and $r$, which in turn becomes immediately obvious on replacing $\bi nj$, $\bi n{j-1}$, and $\bi{j-1}r$ there by the corresponding equal expressions 
$\bi{n+1}j\frac{n+1-j}{n+1}$, $\bi{n+1}j\frac j{n+1}$, and $\bi jr\frac{j-r}j$. 
Theorem~\ref{th:} is now proved. 
\end{proof}

\begin{proof}[Proof of Corollary~\ref{cor:}]
Take any $j\in\intr0n$. Then, in view of \eqref{eq:G_n+1:i} and Theorem~\ref{th:2}, 
\begin{align*}
\E(G_{n+1:j+1}-G_{n+1:j})&=\E\int_0^1 dx\,\ii{G_{n+1:j}\le x<G_{n+1:j+1}} \\ 
	&=\int_0^1 dx\,\P(G_{n+1:j}\le x<G_{n+1:j+1}) \\ 
	&=(-1)^j\bi{n+1}j
	\sum_{r=0}^j(-1)^r\bi jr
	\int_0^1 dx\,\big(1-(n-r+1)x\big)_+^n \\ 
	&=(-1)^j\bi{n+1}j \frac S{n+1},  
\end{align*}
where 
\begin{align*}
	S&:=\sum_{r=0}^j(-1)^r\bi jr\frac1{n-r+1} \\
	&=\sum_{r=0}^j(-1)^r\bi jr\int_0^1 du\,u^{n-r} \\
	&=\int_0^1 du\,u^n\sum_{r=0}^j (-u)^{-r}\bi jr \\
	&=\int_0^1 du\,u^n(1-1/u)^j
	=(-1)^j\frac{(n-j)!j!}{(n+1)!}. 
\end{align*}
So, 
\begin{equation*}
	\E G_{n+1:j+1}-\E G_{n+1:j}=\E(G_{n+1:j+1}-G_{n+1:j})=\frac1{(n+1)(n+1-j)}. 
\end{equation*}
Also, again in view of \eqref{eq:G_n+1:i}, $G_{n+1:0}=0$ and hence $\E G_{n+1:0}=0$. 
Thus, 
\begin{equation*}
	\E G_{n+1:k}=\sum_{j=0}^{k-1}(\E G_{n+1:j+1}-\E G_{n+1:j})
	=\sum_{j=0}^{k-1}\frac1{(n+1)(n+1-j)}=\frac{H_{n+1}-H_{n+1-k}}{n+1}, 
\end{equation*}
which completes the proof of Corollary~\ref{cor:}. 
\end{proof}

The following alternative proof of Corollary~\ref{cor:} is more direct, as it does not rely on Theorem~\ref{th:2}. Instead, it uses the more elementary Remark~\ref{rem:exp}. 

\begin{proof}[``Direct'' proof of Corollary~\ref{cor:}]
Let $X_1,\dots,X_{n+1}$ be as in Remark~\ref{rem:exp}, and then let $X_{n+1:1}\le\cdots\le X_{n+1:n+1}$ be the corresponding order statistics. Then for the r.v.'s $R_1,\dots,R_{n+1}$ defined by \eqref{eq:R_i} and the corresponding order statistics $R_{n+1:1}\le\cdots\le R_{n+1:n+1}$ we have 
\begin{equation}\label{eq:R_n+1,i}
R_{n+1:k}:=\frac{X_{n+1:k}}{X_{n+1:n+1}+\dots+X_{n+1:n+1}}. 	
\end{equation} 
The joint pdf, say $h$, of $X_{n+1:1},\dots,X_{n+1:n+1}$ is given by the formula 
\begin{equation}\label{eq:h}
	h(x_1,\dots,x_{n+1})=(n+1)!e^{-w_{n+1}} 
	\ii{0<x_1<\dots<x_{n+1}}
\end{equation}
for $(x_1,\dots,x_{n+1})\in\R^{n+1}$, 
where 
\begin{equation*}
	w_j:=x_1+\dots+x_j;  
\end{equation*}
see e.g.\ \cite[page~12]{david-nagaraja} again. 

We will also need the following very simple but crucial observation: 
for any real $u>0$
\begin{equation}\label{eq:1/u}
	\frac1u=\int_0^\infty dt\, e^{-tu}. 
\end{equation}
One can view this as a decomposition of the inconvenient function $u\mapsto\frac1u$ into the nice ``harmonics'' $u\mapsto e^{-tu}$. 

Now, introducing  
\begin{equation*}
		L_{n+1,j}:=\int_{S_{n+1}} dx_{1,n+1}\,x_j\, e^{-w_{n+1}},  
\end{equation*}
where 
\begin{equation*}
	S_j:=\{(x_1,\dots,x_j)\colon 0<x_1<\dots<x_j\} \quad \text{and}\quad 
	dx_{1,j}:=dx_1\cdots dx_j 
\end{equation*}
for natural $j$, and using Remark~\ref{rem:exp}, \eqref{eq:R_n+1,i}, \eqref{eq:h}, and \eqref{eq:1/u} (with $u=w_{n+1}$), we can write 
\begin{align}
	\E G_{n+1:k}&=\E R_{n+1:k} \notag \\ 
	&=(n+1)!\int_{S_{n+1}} dx_{1,n+1}\, e^{-w_{n+1}}\,\frac{x_k}{w_{n+1}} \notag \\ 
	&=(n+1)!\int_0^\infty dt\,\int_{S_{n+1}} dx_{1,n+1}\,e^{-w_{n+1}}\,x_k\, e^{-t w_{n+1}} \notag \\ 
		&=(n+1)!\int_0^\infty dt\,\int_{S_{n+1}} dx_{1,n+1}\,x_k\, e^{-(1+t)w_{n+1}} \notag \\ 
	&=(n+1)!\int_0^\infty 
	\frac{dt}{(1+t)^{n+2}}\,L_{n+1,k}  
	=n! L_{n+1,k}.  \label{eq:key}
\end{align}

It remains to evaluate $L_{n+1,k}$. Toward this end, for positive real $t_1,\dots,t_{n+1}$ consider 
\begin{align}
	M(t_1,\dots,t_{n+1})&:=\int_{S_{n+1}} dx_{1,n+1}\, e^{-t_1x_1-\dots-t_{n+1}x_{n+1}} \notag \\ 
&=\frac1{t_{n+1}}\,\int_{S_n} dx_{1,n}\, e^{-t_1x_1-\dots-t_{n-1}x_{n-1}-(t_{n+1}+t_n)x_n} \notag \\ 
\vdots \notag \\ 
&=\frac1{t_{n+1}(t_{n+1}+t_n)\dots(t_{n+1}+\dots+t_1)}. \notag
\end{align}
Therefore, 
\begin{align}
L_{n+1,k}&=-\frac d{dh}\, M(\underbrace{1,\dots,1}_{k-1},1+h,\underbrace{1,\dots,1}_{n+1-k})\bigg|_{h=0} \notag \\ 
&=-\frac d{dh}\, \frac1{(n+1-k)!(n+2-k+h)\cdots(n+1+h)}\bigg|_{h=0} \notag \\ 
&=\frac1{(n+1)!}\Big(\frac1{n+2-k}+\dots+\frac1{n+1}\Big) 
=\frac{H_{n+1}-H_{n+1-k}}{(n+1)!}. \notag  
\end{align}
Now Corollary~\ref{cor:} follows by \eqref{eq:key}. 
\end{proof}


\bibliographystyle{amsplain}
\bibliography{P:/pCloudSync/mtu_pCloud_02-02-17/bib_files/citations10.13.18a}

\end{document}